\documentclass[letterpaper, 10 pt, conference]{ieeeconf}  
\IEEEoverridecommandlockouts                              
\let\labelindent\relax

\usepackage{graphicx}
\usepackage{epsfig}
\usepackage{mathptmx}
\usepackage{times}
\usepackage{amsmath} 
\usepackage{amsthm}
\usepackage{amssymb}
\usepackage{tcolorbox}
\usepackage{cite}
\usepackage[ruled,vlined]{algorithm2e}
\usepackage[colorlinks = true, linkcolor = blue, citecolor =
blue]{hyperref}
\usepackage{enumitem}
\usepackage{subcaption}
\usepackage[mathscr]{euscript}
\usepackage{xcolor}
\usepackage{tabularx}

\newtheorem{thm}{Theorem}
\newtheorem{prop}[thm]{Proposition}
\newtheorem{lem}[thm]{Lemma}

\newtheorem{rem}[thm]{Remark}

\newcommand{\ldef}{:=}

\newcommand{\real}{\ensuremath{\mathbb{R}}}

\newcommand{\ints}{{\mathbb{Z}}}
\newcommand{\nat}{{\mathbb{N}}}
\newcommand{\natz}{{\mathbb{N}}_0}

\newcommand{\norm}[1]{\left\lVert #1 \right\rVert}

\DeclareMathOperator*{\argmin}{arg\,min}

\newcommand{\tth}{^\text{th}}
\newcommand{\param}{\textbf{a}}
\newcommand{\sys}{\mathscr{S}}

\newcommand{\thmtitle}[1]{\mbox{}\textit{(#1).}}

\renewcommand{\theenumi}{(\arabic{enumi})}

\newcommand{\remend}{\relax\ifmmode\else\unskip\hfill\fi\hbox{$\bullet$}}

\title{\LARGE \bf
	A Control Lyapunov Function Approach to Event-Triggered Parameterized Control for Discrete-Time Linear Systems
}

\author{Anusree Rajan, Kushagra Parmeshwar, and Pavankumar Tallapragada
  \thanks{Authors acknowledge the Centre for Networked Intelligence (a Cisco CSR initiative) at Indian Institute of Science for funding this work.
  	Anusree Rajan and Pavankumar Tallapragada are with the Department of Electrical 
  	Engineering, Indian Institute of Science, Bengaluru. Kushagra Parmeshwar is with the Department of Electrical 
  	Engineering, Indian Institute of Technology, Kharagpur.  {\tt\small anusreerajan@iisc.ac.in, kush8725@gmail.com, pavant@iisc.ac.in}}
}

\renewcommand{\baselinestretch}{1}
\begin{document}
	
	\maketitle
	\thispagestyle{empty}
	\pagestyle{empty}

\begin{abstract}
This paper proposes an event-triggered parameterized control method using a control Lyapunov function approach for discrete time linear systems with external disturbances. In 
this control method, each control input to the plant is a linear combination of a fixed set of linearly independent scalar functions. The controller updates the coefficients of the parameterized 
control input in an event-triggered manner so as to 
minimize a quadratic cost function subject to quadratic constraints and communicates the same to the actuator. We design an 
event-triggering rule that guarantees global uniform ultimate 
boundedness of trajectories of the closed loop system and non-trivial inter-event times. We illustrate our 
results through numerical examples and we also compare the performance of the proposed control method with other existing control methods in the literature.
\end{abstract}

\section{INTRODUCTION}

Event-triggered control (ETC) is a promising control method, especially in networked control systems, due to its efficient utilization of resources compared to the classical time-triggered control method. Recent studies in the ETC literature try to explore the possibility of further improving the efficiency of resource utilization by designing control laws based on non-zero order hold (non-ZOH) techniques instead of the popular ZOH technique. However, most of the existing ETC methods based on non-ZOH control either require more computational capacity at the actuator or require transmitting a larger amount of information over the communication network at each communication time instant. An exception to this is the event-triggered parameterized control (ETPC) method proposed in~\cite{AR-PT:2023}. In this paper, we extend this idea using a control Lyapunov function (CLF) method for discrete-time linear systems with external disturbances. This is in contrast to the emulation based approach, which is far more common in event-triggered control literature.

\subsection{Literature Review}

A fundamental overview of the ETC method, along with relevant literature, is discussed in~\cite{ PT:2007, WH:2012, ML:2010, DT-SH:2017-book}. Generally, in ETC and in other closely
related approaches, such as self-triggered control~\cite{AA:2010} and periodic event-triggered control~\cite{WH:2013}, the control input to the plant is held constant between any two
consecutive triggering instants. However, there are some exceptions to this basic approach. For example, in model-based ETC~\cite{EG-PA:2013, MH-FD:2013, HZ-etal:2016, ZC-etal:2021,
  LZ-etal:2021}, a time-varying control input is applied to the plant even between two successive events by using a model of the plant at the actuator. In event/self-triggered model
predictive control (MPC)~\cite{HL-YS:2014, FD-MH-FA:2017, HL-etal:2018}, at each triggering instant, the controller generates a control trajectory generated by solving a finite horizon
optimization problem and then transmits it to the actuator, and the actuator applies the same to the plant until the next event. As discussed in~\cite{AL-JS:2023,KH-etal:2017}, the
efficiency of communication resource utilization in the model-predictive control method can be improved by transmitting only some of the samples of the generated control trajectory to
the actuator, based on which a sampled data first-order-hold (FOH) control input is applied to the plant. Another example of a non-ZOH-based ETC method is event-triggered dead-beat
control~\cite{BD-etal:2017}, where a sequence of control inputs is transmitted to the actuator in an event-triggered manner and the same is applied to the plant till the next packet is
received. 

Our recent work~\cite{AR-PT:2023} proposes a novel non-ZOH based ETC method, called as event-triggered parameterized control (ETPC) method, 
for stabilization of linear systems. In~\cite{AR-PT:2024}, we extend this control method to nonlinear control settings with external disturbances. In~\cite{AR-etal:2023}, we use a 
similar idea to design an event-triggered polynomial controller for 
trajectory tracking by unicycle robots. In all these works, we use an emulation based approach for determining the parameters at each event-triggering instant. There are also a few papers that use a parameterized control law in MPC like problems but not with even-triggering. For example, in our recent work~\cite{AR-AK-PT:2024}, we co-design a polynomial control law and a communication scheduling strategy for multi-loop networked control systems. Another example is~\cite{SD-eal:2023} which introduces a numerical algorithm that
serves as a preliminary step toward solving continuous-time
MPC problems directly without
explicit time-discretization.

\subsection{Contributions} \label{sec:contribs}

The contributions of this paper are given below:

\begin{itemize}
\item We design an event-triggered parameterized control law for discrete-time linear systems with external disturbances, using a control Lyapunov function approach. At each event, a
  parametrized control trajectory is generated by optimizing a quadratic cost in the state and control signals. This is in contrast to much of the literature on ETC, which employs an
  approach wherein a continuous feedback controller is emulated by ETC. For our proposed method, we guarantee global uniform ultimate boundedness of trajectories of the closed loop
  system and non-trivial inter-event times.
  
\item Compared to the model-based control method, the proposed parameterized control method 
requires less computational resources at the actuator and also 
provides greater privacy and security. 
\item Compared to the MPC-based control method, at each event, our proposed method requires only a limited 
number of parameters to be sent irrespective of the time duration of 
the signal. 
\item In this paper, we extend the control method proposed in our previous work~\cite{AR-PT:2023, AR-PT:2024, AR-etal:2023} to design an optimal control law for discrete-time linear
  systems with external disturbances. In~\cite{AR-PT:2023}, design of the proposed parameterized control law is based on a two stage process - generation of an ideal continuous time
  feedback control signal by simulating the system for some time duration and then optimally approximating the ideal feedback control signal. In this paper, directly obtain an optimal parametrized control signal by using the control Lyapunov function approach.

	\end{itemize}

%

\subsection{Notation}

Let $\real$ denote the set of all real numbers. Let $\ints$, $\nat$ and $\natz$ denote the set of all integers, positive and non-negative integers, respectively. For $a, b \in \real$, we
let $[a, b]_\ints := [a, b] \cap \ints$ and $[a, b)_\ints := [a, b) \cap \ints$. For any $x \in \real^n$, $\norm{x}$ denotes the euclidean norm. For a square matrix
$A \in \real^{n \times n}$ with real eigenvalues, let $\lambda_{\min}(A)$ and $\lambda_{\max}(A)$ denote the smallest and the largest eigenvalues of A, respectively. Further, for a symmetric matrix
$A \in \real^{n \times n}$, $A \succ 0$, $A \succeq 0$ and $A \prec 0$ mean that $A$ is positive definite, positive semi-definite and negative definite, respectively.

\section{PROBLEM SETUP}	\label{sec:problem_setup}

In this section, we present the system dynamics, the parameterized control law and the objective of this paper.

\subsection*{System Dynamics and Control Law}

Consider a discrete-time linear time-invariant system with external disturbance,
\begin{equation}\label{eq:sys}
x(t+1)=Ax(t) + Bu(t)+d(t), \quad \forall t \in \natz,
\end{equation}
where $x\in \real^n$, $u \in \real^m$, and $d \in \real^q$, respectively, denote the system state, the control input, and the 
external disturbance. 
\begin{enumerate}[resume, label=\textbf{(A\arabic*)},align=left]
	\item We assume that there exists $D\ge0$ such that $\norm{d(t)}\le D, \ \forall t \in \natz$. \label{A:d}
\end{enumerate}

In this paper, we consider a parameterized control law where each control input to the plant is a linear combination of a set of linearly independent scalar functions. The coefficients of the parameterized control input are updated in an event-triggered manner.

Specifically, we consider a set of functions
\begin{equation*}
\Phi := \left\{ \phi_j: [0, \infty)_{\ints} \to \real 
\right\}_{j=0}^p,
\end{equation*}
which satisfies the following standing assumption.

\begin{enumerate}[resume, label=\textbf{(A\arabic*)},align=left]
	\item $\Phi$ is a set 
	of linearly independent functions when restricted to $[0, N]_{\ints}$ where $N \in \nat$ is a fixed parameter, i.e.,
	$ \sum_{j=0}^{p}c_j\phi_j(t)=0,$ $\forall t \in [0, N]_{\ints}$ iff $c_j = 0,$ $\forall j \in \{0,1,\ldots,p\}$.
	\label{A:phis}
\end{enumerate}

Then, we consider the following control law,
\begin{equation}\label{eq:control_law}
u(t_k + \tau) = \mathbb{P}(\tau)\param(k), \ \forall \tau \in [0, t_{k+1} - t_k)_\ints,
\end{equation}
where \begin{equation*}
\mathbb{P}(\tau) \ldef \begin{bmatrix}
\phi^{\top}(\tau) & 0 & \ldots & 0 \\
0 & \phi^{\top}(\tau) & \ldots & 0\\
\vdots & \vdots & \ldots & \vdots\\
0 & 0 & \ldots & \phi^{\top}(\tau)
\end{bmatrix} \in \real^{m \times m(p+1)} ,
\end{equation*} and
$\phi^\top(\tau) \ldef \begin{bmatrix}
\phi_0(\tau) & \phi_1(\tau)& \ldots & \phi_p(\tau)
\end{bmatrix}$. Here, $\param(k) \in \real^{m(p+1)}$ is a column vector which contains the coefficients of the parameterized control law. $(t_k)_{k \in \natz}$ denotes the sequence of time instants at which the controller computes the coefficients of the parameterized control law and communicates them to the actuator.

The general configuration of the event-triggered parameterized control system considered in this paper is depicted in Figure~\ref{fig:ETPC_system}.
\begin{figure}[h]
	\centering
	\includegraphics[width=7cm]{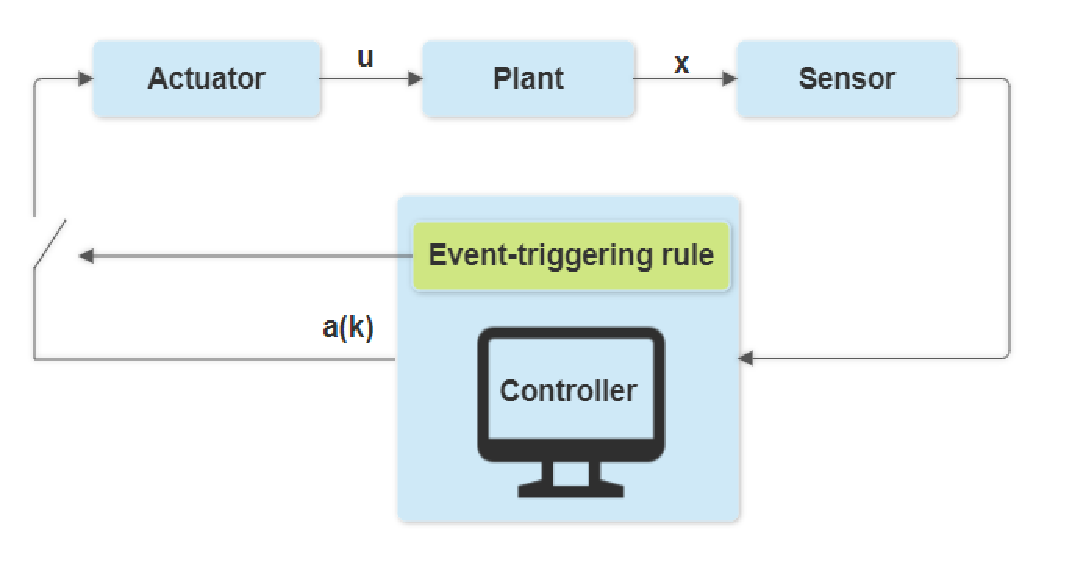}
	\caption{Event-triggered parameterized control configuration}
	\label{fig:ETPC_system}
\end{figure} 
Here, the system state is continuously available to the 
controller which has enough computational resources to evaluate the 
event-triggering condition and to update the coefficients of the 
control input at an event-triggering instant. 

\subsection*{Objective}

We seek to design a parameterized control law~\eqref{eq:control_law} using a control Lyapunov function approach, instead of the much more commonly used emulation based
approach, and an event-triggering rule for implicitly determining the communication instants $(t_k)_{k \in \natz}$ so that the trajectories of the closed loop system are globally
uniformly ultimately bounded.

\section{DESIGN OF EVENT-TRIGGERED CONTROLLER}\label{sec:design}
In this section, we design a parameterized control law and an 
event-triggering rule to achieve our objective.

 \subsection{Design of Parameterized Control Law}

At each triggering instant $t_k$, the controller determines the new 
coefficients $\param(k)$ by solving the following finite 
horizon optimization problem,
\begin{equation}\label{eq:a_k}
\begin{aligned}
\param(k)  \in & \argmin_{a \in \real^{m(p+1)}} \quad \sum_{t=t_k}^{t_k+N}
\left[V(\hat{x}(t))+u(t)^TRu(t)\right],\\
& \textrm{s.t.} \ \
\hat{x}(t+1) =  A\hat{x}(t)+Bu(t), \ \hat{x}(t_k)=x(t_k),\\ & u(t)=\mathbb{P}(t-t_k)a, \ \forall t \in [t_k,t_k + N]_\ints, 
\\ & V(\hat{x}(t)) \le \alpha^{t-t_k} V(\hat{x}(t_k)),
\ \forall t \in [t_k,t_k + M]_\ints.
\end{aligned}
\end{equation} 
Here, $M \le N \in \nat$ and $\alpha \in (0,1)$ are design parameters. $R \succeq 0$ and $V(x) \ldef x^{\top}Px$ is a Lyapunov-like function with $P \succ 0$.

Note that, given the dynamics, we can write the closed form expression for $\hat{x}(t)$ as follows,
\begin{equation*}
\hat{x}(t)=F(t-t_k)x(t_k)+G(t-t_k)a, \ \forall t \in [t_k,t_k+N]_{\ints},
\end{equation*}
where $F(\tau) \ldef A^{\tau}$ and $G(\tau) \ldef \sum_{j=0}^{\tau-1}A^{\tau-1-j}B\mathbb{P}(j)$. By using the above expression, we can rewrite the optimization problem~\eqref{eq:a_k} as the following
quadratically constrained quadratic optimization problem,

\begin{equation}\label{eq:QCQP}
\begin{aligned}
\param(k)  \in & \argmin_{a \in \real^{m(p+1)}} \quad J(a),
\\
& \textrm{s.t.} \ \
H_{\tau}(a) < 0, \ \forall \tau \in [0,M]_{\ints},
\end{aligned}
\end{equation} 
where 
\begin{align*}
  J(a) = \ & a^{\top}\left[\sum_{\tau=0}^{N}\left(G^{\top}(\tau)PG(\tau)+\mathbb{P}^{\top}(\tau)R\mathbb{P}(\tau)\right)\right]a
  \\
  & + 2x^{\top}(t_k)\left[\sum_{\tau=0}^{N}F^{\top}(\tau)PG(\tau)\right]a
  \\
  & + x^{\top}(t_k)\left[\sum_{\tau=0}^{N}F^{\top}(\tau)PF(\tau)\right]x(t_k),
\end{align*}
\begin{align*}
H_{\tau}(a)  = & \  a^{\top}G^{\top}(\tau)PG(\tau)a+2x^{\top}(t_k)F^{\top}(\tau)PG(\tau)a \\ & + x^{\top}(t_k)\left[F^{\top}(\tau)PF(\tau)-\alpha^{\tau}P\right]x(t_k).
\end{align*}

\begin{rem}\label{rem:convexity}
	The quadratically constrained quadratic optimization problem~\eqref{eq:QCQP} is a convex 
	optimization problem as $G^{\top}(\tau)PG(\tau) \succeq 0$ and $\mathbb{P}^{\top}(\tau)R\mathbb{P}(\tau) \succeq 0$ for all $\tau \in [0,N]_{\ints}$. \remend 
\end{rem}
\begin{rem}\label{rem:CLF-emulation}
  In the emulation based approach used in~\cite{AR-PT:2023,AR-PT:2024,AR-etal:2023}, we first have to simulate the system for some time duration in the future and then optimally
  approximate the ideal feedback control signal using a parametrized control. In contrast, in this paper, we use a control Lyapunov function approach for designing the parameterized
  control law. The major advantages of this approach are that it directly optimizes the control trajectory in the space of the parametrized functions and we can also incorporate a
  greater variety of cost functions and constraints in the problem. \remend
\end{rem}
\begin{rem}\label{rem:comparison-ETMPC}
  Compared to the event-triggered model predictive control and dead-beat control, the ETPC method allows for having inter-event times strictly greater than the prediction horizon length
  $N$. Once the coefficients $\param(k)$ are determined at $t_k$, the control law $u(t)$ is well defined for the interval $[t_k,t_{k+1})_{\ints}$ even if $t_{k+1}-t_k>N$. In addition,
  in ETPC, since only the parameters of the control signal need to be communicated, the communication load is significantly decreased even for long horizons $N$. \remend
\end{rem}

Next we provide a sufficient condition to ensure the feasibility of~\eqref{eq:QCQP}.

\begin{prop}\thmtitle{Sufficient condition to ensure the feasibility of~\eqref{eq:QCQP}}\label{prop:feasibility}
	The optimization problem~\eqref{eq:QCQP} is feasible if there exists a solution $C \in \real^{m(p+1) \times n}$ for the following linear matrix inequality (LMI), $ \forall \tau \in [0,M]_{\ints}$,
	\begin{equation}\label{eq:LMI}
L_0(\tau) + \sum_{i,j} c_{ij} L_{ij}(\tau) \succ 0,
	\end{equation}
	where $c_{ij}$ denotes the $\{i,j\}^{\tth}$ element of $C$, 
	\begin{equation*}
          L_0 (\tau) = \begin{bmatrix}
            \alpha^{\tau} P^{-1} & (F(\tau)P^{-1})^\top \\
            F(\tau)P^{-1} & P^{-1}
	\end{bmatrix},
	\end{equation*}
	and
\begin{equation*}
  L_{ij} = \begin{bmatrix}
    0 & (Q_{ij}(\tau))^\top \\
    Q_{ij}(\tau) & 0
\end{bmatrix} .
\end{equation*}
Here, \(Q_{ij}(\tau)\) is the matrix formed by multiplying the \(i{\tth}\) column of \(G(\tau)\) with the \(j{\tth}\) row of \(P^{-1}\).
\end{prop}

\begin{proof}
	First, note that, we can rewrite the linear matrix inequality~\eqref{eq:LMI} in the following matrix form,
	\begin{equation*}
	\begin{bmatrix}
	\alpha^{\tau} P^{-1} & \left[\left(F(\tau) + G(\tau)C\right)P^{-1}\right]^\top \\
	\left(F(\tau) + G(\tau)C\right)P^{-1} & P^{-1}
	\end{bmatrix} \succ 0.
	\end{equation*}
	Now, by using Schur complement lemma, we can say that the above inequality is true if and only if,
	\begin{equation*}
          \alpha^{\tau} P^{-1} - \left[\left(F(\tau) + G(\tau)C\right)P^{-1}\right]^{\top}P\left[\left(F(\tau) + G(\tau)C\right)P^{-1}\right] \succ 0.
	\end{equation*} 
	This implies that,
		\begin{equation*}
	\left[F(\tau) + G(\tau)C\right]^{\top}P\left[F(\tau) + G(\tau)C\right] - \alpha^{\tau} P \prec 0.
	\end{equation*} 
	If there exists a $C \in \real^{m(p+1) \times n}$ which satisfies the above inequality for all $\tau \in [0,M]_{\ints}$, then we can say that $a=Cx(t_k)$ is a feasible solution of the optimization problem~\eqref{eq:QCQP}.
\end{proof}

\begin{rem}\label{rem:stabilizability}
  Assume that $\phi_0$ is a non-zero constant function and $\phi_j(0)=0, \ \forall j \in \{1,2,\ldots,p\}$. If the pair $(A,B)$ is controllable, then there always exists an $M \in \nat$ and
  $C \in \real^{m(p+1) \times n}$ that satisfy the LMI~\eqref{eq:LMI} for the choice of $P \succ 0$ which is a solution of the Lyapunov equation $(A+BK)^{\top}P(A+BK)-P=-Q$ and
  $\alpha \in [1-\frac{\lambda_{\min}(Q)}{\lambda_{\max}(P)},1),$ for some $Q \succ 0$ and $K \in \real^{m \times n}$ such that $A+BK$ is Schur stable and which satisfies the desired convergence rate
  constraint. Specifically $C \in \real^{m(p+1) \times n}$ such that $\mathbb{P}(0)C=K$ is guaranteed to satisfy the LMI~\eqref{eq:LMI} for $M=1$. \remend
\end{rem}

\subsection{Design of Event-Triggering Rule}
Next, we design an event-triggering rule that implicitly determines the time instants at which the controller updates the coefficients of the parameterized control input and communicates the same to the actuator. But first, we define the following \emph{predictor} function
\begin{align*}
\bar{V}(t+1 | t_k) := & [Ax(t)+Bu(t |t_k)]^{\top}P[Ax(t)+Bu(t |t_k)] \\ & +\lambda_{\max}(P) \left(D^2+2D\norm{Ax(t)+Bu(t |t_k)} \right),
\end{align*}
where $u( t | t_k ) := \mathbb{P}( t - t_k) \param(k)$ is the control trajectory computed at $t_k$. As we will see in the next result, the
predictor function provides an upper bound on $V(x(t+1))$, over all possible disturbances, if at time $t$ the control input $u(t) = u(t | t_k)$. Thus, the predictor function could help us evaluate the necessity of replanning and updating the control trajectory at each timestep.
\begin{lem}\label{lem:V-bound}
  For all $t \in [t_k , t_{k+1}]_\ints$, and $\forall k \in \natz$, if $u(t) = u(t | t_k)$ then $V(x(t+1)) \leq \bar{V}(t+1 | t_k)$.
\end{lem}
\begin{proof}
  For any $t \in [t_k , t_{k+1}]_\ints$, and $\forall k \in \natz$, if $u(t) = u(t | t_k)$, then we can say that
\begin{align*}
& V(x(t+1)) = x^{\top}(t+1)Px(t+1)
\\ & = [Ax(t)+Bu(t | t_k)+d(t)]^\top P [Ax(t)+Bu(t | t_k)+d(t)].
\end{align*}
Simplifying this, and by using Assumption~\ref{A:d}, we obtain
\begin{align*}
V(x(t+1)) &
= \ [Ax(t)+Bu(t | t_k)]^{\top}P[Ax(t)+Bu(t | t_k)] + \\ & d^{\top}(t)Pd(t)+2 d^{\top}(t)P[Ax(t)+Bu(t | t_k)]
\\ &
= \ [Ax(t)+Bu(t | t_k)]^{\top}P[Ax(t)+Bu(t | t_k)] + \\ & \lambda_{\max}(P)(\norm{d(t)}^2+2 \norm{d(t)}\norm{Ax(t)+Bu(t | t_k)})
\\ &
\le \bar{V}(t+1 | t_k).
\end{align*}
\end{proof}
Note that as per~\eqref{eq:control_law}, $u(t) = u(t | t_k)$, for all $t \in [t_k , t_{k+1})_\ints$, and $\forall k \in \natz$. However, in Lemma~\ref{lem:V-bound}, we consider the hypothetical
scenario $u(t) = u(t | t_k)$ for $t = t_{k+1}$. This is because at $t_{k+1}$, in order to first decide if a replanning of the control trajectory is required at $t_{k+1}$, we need to
evaluate the usefulness and the likely effect of the previously computed input, $u(t | t_k)$.

Now, we present the event-triggering rule below
\begin{align}
  t_{k+1} &= \min \left\lbrace t>t_k : \bar{V}(t+1 | t_k) > H( t, t_k) \right\rbrace, \label{eq:ETR}
  \\
  H(t , t_k) &:= \max \left\lbrace\epsilon^2, \beta^{t-t_k+1}V(x(t_k))\right\rbrace , \label{eq:H}
\end{align}
where $t_0=0$ and $\epsilon \ldef \frac{D}{\sigma}$. Here $\beta \in (0,1)$ and $\sigma >0$ are design parameters.

In summary, the closed loop system, $\sys$, is the combination 
of the system dynamics~\eqref{eq:sys}, the parameterized control 
law~\eqref{eq:control_law}, with coefficients chosen by 
solving~\eqref{eq:QCQP}, which are updated at the events determined by 
the event-triggering rule~\eqref{eq:ETR}. That is,
\begin{equation}\label{eq:full-system}
\sys : \ \eqref{eq:sys}, \eqref{eq:control_law}, 
\eqref{eq:QCQP}, \eqref{eq:ETR}.
\end{equation}

\section{ANALYSIS OF THE EVENT-TRIGGERED CONTROLLER}\label{sec:analysis}
In this section, we analyze the proposed event-triggered parameterized controller.  We first present a lemma that helps to prove the main result of this paper.
\begin{lem}\label{lem:V}
  Consider the closed loop system~\eqref{eq:full-system}. If $0 < \alpha < \beta < 1$ and $\sigma \le \bar{\sigma}$ where
  \begin{equation*}
    \bar{\sigma} \ldef \displaystyle \min_{\tau \in [1,M]_{\ints}} \left\lbrace\frac{ - \sqrt{ \frac{\alpha^{\tau}}{\lambda_{\min}(P)} } + \sqrt{ \frac{\alpha^{\tau}}{\lambda_{\min}(P)}+\frac{\beta^{\tau}-\alpha^{\tau}}{\lambda_{\max}(P)} } }{ 1+\norm{A}\bar{A}(\tau-1) } \right\rbrace ,
  \end{equation*}
  $\bar{A}(\tau) \ldef \norm{\sum_{j=0}^{\tau-1}A^j}$ with $\bar{A}(0)=0$, and $M \in \nat$ is same as in~\eqref{eq:a_k}, then the following statements are true.
  \begin{itemize}
  \item If $V(x(t_k)) \ge \epsilon^2$, for some $k \in \natz$, then $\bar{V}(t_k+\tau | t_k) \le \beta^{\tau} V(x(t_k))$, $\forall \tau \in [1,M]_{\ints}$.
  \item If $V(x(t_k)) \le \epsilon^2$, for some $k \in \natz$, then $\bar{V}(t_k+\tau | t_k) \le \epsilon^2$, $\forall \tau \in [1,M]_{\ints}$.
  \end{itemize}
\end{lem}

\begin{proof}
  First, consider the function
  \begin{align*}
    \gamma( \tau ) := &\left(\alpha^{\tau}+ \lambda_{\max}(P) \left(1+\norm{A}\bar{A}(\tau-1)\right)^2\sigma^2\right) +
    \\ & 2\lambda_{\max}(P)\left(1+\norm{A}\bar{A}(\tau-1)\right)\sqrt{\frac{\alpha^\tau }{\lambda_{\min}(P)}}\sigma .
  \end{align*}
  Then $\sigma \le \bar{\sigma}$ and the definition of $\bar{\sigma}$ imply that
  \begin{equation}
    \label{eq:gamma-bound}
    \gamma( \tau ) \leq \beta^\tau , \quad \forall \tau \in [1,M]_{\ints} .
  \end{equation}
          
  Now note that, we can rewrite the definition of $\bar{V}(t+1 | t_k)$ as follows,
  \begin{align*}
    \bar{V}(t+1 | t_k) = & [\hat{x}(t+1 | t_k)+Ae(t | t_k)]^{\top}P[\hat{x}(t+1 | t_k)+Ae(t | t_k)] \\ & +\lambda_{\max}(P)(D^2+2D\norm{\hat{x}(t+1 | t_k)+Ae(t | t_k)})
  \end{align*}
  where $e(t | t_k) \ldef x(t)-\hat{x}(t | t_k)$ and $\hat{x}(t | t_k)$ is the nominal state trajectory which follows the dynamics
  \begin{equation*}
    \hat{x}(t+1 | t_k) =  A\hat{x}(t | t_k)+Bu(t | t_k), \ \hat{x}(t_k | t_k)=x(t_k).
  \end{equation*}
   
 Then, $\forall \tau \in [1,M]_{\ints}$,
  \begin{align*}
    &\bar{V}(t_k+\tau | t_k) =
    \\
    &V(\hat{x}(t_k+\tau | t_k)) + 2\hat{x}^{\top}(t_k+\tau | t_k)PAe(t_k+\tau-1 | t_k) +
    \\
    &e^{\top}(t_k+\tau-1 | t_k)A^{\top}PAe(t_k+\tau-1 | t_k) + \\ & \lambda_{\max}(P) (D^2+2D\norm{\hat{x}(t_k+\tau | t_k)+Ae(t_k+\tau-1 | t_k)}).
   \end{align*}
   Note that, for any $t \in [t_k , t_{k+1}]_\ints$, and $\forall k \in \natz$, if $u(t) = u(t | t_k)$, then $e(t+1 | t_k)= A e(t | t_k) + d(t)$ and hence $\norm{e(t_k+\tau | t_k)} \le \bar{A}(\tau)D$.  
  By using the fact that, for any $k \in \natz$ and $\forall \tau \in [0,M]_{\ints}$, $V(\hat{x}(t_k+\tau | t_k)) \le \alpha^{\tau}V(x(t_k))$ from the constraints in~\eqref{eq:a_k}, we can say that
	\begin{align*}
	\bar{V}(t_k&+\tau | t_k)  \le
	\alpha^{\tau} V(x(t_k)) + \lambda_{\max}(P) \left(\norm{A}\bar{A}(\tau-1)D\right)^2+\\&2\lambda_{\max}(P)\sqrt{\frac{\alpha^\tau V(x(t_k))}{\lambda_{\min}(P)}}\norm{A}\bar{A}(\tau-1)D \ +\\& \lambda_{\max}(P)\left[D^2+2D\left(\sqrt{\frac{\alpha^\tau V(x(t_k))}{\lambda_{\min}(P)}}+\norm{A}\bar{A}(\tau-1)D\right)\right].
	\end{align*}
	This implies that,
        \begin{align}
          \bar{V}(t_k+\tau | t_k)  \le & 
                                    \alpha^{\tau} V(x(t_k)) + \lambda_{\max}(P) D^2\left(1+\norm{A}\bar{A}(\tau-1)\right)^2+ \notag
          \\ & 2\lambda_{\max}(P)D\left(1+\norm{A}\bar{A}(\tau-1)\right)\sqrt{\frac{\alpha^\tau V(x(t_k))}{\lambda_{\min}(P)}}. \label{eq:Vbar-bound}
	\end{align}
  Note that, in the first statement of this lemma, as $D^2 \le \sigma^2 V(x(t_k))$ we can say from~\eqref{eq:gamma-bound} that
  \begin{equation*}
    \bar{V}(t_k+\tau | t_k) \le \gamma( \tau ) V(x(t_k))  \le \beta^\tau V(x(t_k)), \quad \forall \tau \in [1,M]_{\ints}.
  \end{equation*}
  This completes the proof of the first statement of this lemma.
  
  Next note that, in the second statement of this lemma, as $V(x(t_k)) \le \frac{D^2}{\sigma^2}$, we can say from~\eqref{eq:Vbar-bound} that
  \begin{equation*}
    \bar{V}(t_k+\tau | t_k) \le \gamma( \tau ) \frac{D^2}{\sigma^2} \le \ \beta^\tau \frac{D^2}{\sigma^2} \ \le \ \frac{D^2}{\sigma^2}, \quad \forall \tau \in [1,M]_{\ints}.
\end{equation*}
where the last inequality follows from the fact that $\beta^\tau < 1, \ \forall \tau \in [1,M]_{\ints}$. This completes the proof of the second statement of this lemma.
\end{proof}

Next, we present the main theorem of this paper. 
\begin{thm}\label{thm:GUUB}\thmtitle{Lower bound on inter-event times and global uniform ultimate boundedness of trajectories}
  Consider the closed loop system~\eqref{eq:full-system}. Let $M \ge 1$ in~\eqref{eq:a_k} and let the conditions of Lemma~\ref{lem:V} be satisfied. Then,
  \begin{itemize}
  \item The inter-event times $t_{k+1}-t_k \ge M$, $\forall k \in \natz$ and if $M \ge 2$ then the inter-event times are non-trivial, i.e., $t_{k+1}-t_k > 1$, $\forall k \in \natz$.
  \item If $V(x(t_k)) \le \epsilon^2$ for some $k \in \natz$, then $V(x(t)) \le \epsilon^2$, $\forall t \in [t_k,\infty)_{\ints}$. 
  \item If $V(x(t_0)) > \epsilon^2$, then there exists a $k \in \nat$ such that $V(x(t_k)) \le \epsilon^2$.
  \item The trajectories of the closed loop system~\eqref{eq:full-system} are globally uniformly ultimately bounded with $\epsilon^2$ being the ultimate bound on $V(x)$.
  \end{itemize}
\end{thm}
\begin{proof}
  Let us prove the first statement of this theorem. Note that, according to the event-triggering rule~\eqref{eq:ETR}, an event is triggered at $t>t_k$ if and only if
  $\bar{V}(t+1 | t_k) > \epsilon^2$ and $\bar{V}(t+1 | t_k) > \beta^{t-t_k+1}V(x(t_k))$. Lemma~\ref{lem:V} shows that, for any $t \in [t_k,t_k+M-1]_{\ints}$, at least one of the two conditions given
  above is not satisfied. This implies that $t_{k+1}-t_k \ge M$ for $\forall k \in \natz$. This completes the proof of the first statement.

  Now, let us prove the second statement by contradiction. Let there exist $\bar{t} \in [t_k+1,\infty)_{\ints}$ such that $V(x(\bar{t})) > \epsilon^2$ and $V(x(t)) \le \epsilon^2$ for all
  $t \in [t_k,\bar{t}-1]_{\ints}$. Let $t_q \ge t_k$ be such that $\bar{t} \in [t_q , t_{q+1}]$. Then, by Lemma~\ref{lem:V-bound}, $\bar{V}(\bar{t} | t_q) \ge V(x(\bar{t}))>\epsilon^2$. This implies,
  according to the event-triggering rule~\eqref{eq:ETR}, that an event must be triggered at $t = \bar{t} - 1$, i.e., $t_{q+1} = \bar{t}-1$, i.e.,
  $\bar{t} = t_{q+1} + 1 \notin [t_q, t_{q+1}]_\ints$. Similarly, by using the second statement of Lemma~\eqref{lem:V}, we can say that
  $V(x(\bar{t})) \le \bar{V}(\bar{t} | t_{q+1}) = \bar{V}(t_{q+1} + 1 | t_{q+1}) \le \epsilon^2$, which is a contradiction. Thus, there does not exist such a $\bar{t}$ and this completes the proof
  of the second statement.
	
  Next, we prove the third statement. As $M \ge 1$, according to the event-triggering rule~\eqref{eq:ETR} and Lemma~\ref{lem:V-bound} and Lemma~\ref{lem:V}, if $V(x(t_k)) > \epsilon^2$ for any
  $k \in \natz$, then $V(x(t_{k+1})) \le \max\{ \beta^M V(x(t_k)), \epsilon^2 \}$. Thus $\{V(x(t_k))\}$ is a monotonically decreasing sequence, with a uniform bound $\beta^M < 1$ on the rate of decrease,
  as long as $V(x(t_{k+1})) > \epsilon^2$. Hence, there must exists a $q \in \nat$ such that $V(x(t_q)) \le \epsilon^2$.
	
Now, by using the second and the third statements, we can say that for any initial state $x(t_0)$ there exists a $T \in \natz$ such that $V(x(t)) \le \epsilon^2$ for all $t \in [T,\infty)_{\ints}$. 
This completes the proof of this theorem.
\end{proof}

\section{NUMERICAL EXAMPLES}\label{sec:numerical_examples}
In this section, we present a numerical example to illustrate our results. \newline 
\textbf{Example 1:}
Consider the system,
\begin{equation*}
x(t+1)=\begin{bmatrix}
0.7 & -0.1 & -0.1 \\
0 &   0.8  & -0.4 \\
0  & 0 & 1.2
\end{bmatrix}x(t)+\begin{bmatrix}
0 \\ 0 \\ 1
\end{bmatrix}u(t)+d(t),
\end{equation*}
for all $t \in \natz$. In this example, we consider the control input as 
a linear combination of the set of functions $\{1,\tau, 
\tau^2,\ldots,\tau^p\}$. That is each control input to the plant is a polynomial of degree $p$. We consider the external disturbance $d(t)=\frac{0.01}{\sqrt{3}}\begin{bmatrix}
	\sin(50t) &  \sin(20t) &  \sin(10t)
	\end{bmatrix}^{\top}$ that satisfies Assumption~\ref{A:d} with $D=0.01$. We choose the quadratic Lyapunov function $V(x) \ldef x^{\top}Px$, 
where $P \succ 0$ is chosen such that it satisfies the Lyapunov 
equation $(A+BK)^{\top}P(A+BK)-P=-Q,$
with $Q=0.01\mathbb{I}$ where $\mathbb{I}$ is a $3 \times 3$ identity matrix and
$K= \begin{bmatrix}
0 & 0 & -0.3
\end{bmatrix}$.
According to Proposition~\ref{prop:feasibility} and Remark~\ref{rem:stabilizability}, we can verify that the optimization problem~\eqref{eq:QCQP} has a feasible solution for any $M \in [1,8]_{\ints}$. We choose the design parameters $M=2$, $R=1$, $\alpha=0.952, \ \beta=0.99$, and $\sigma=0.01$ which satisfy the conditions given in Lemma~\ref{lem:V}.

We compare the performance of the proposed CLF based ETPC method (ETPC-CLF) with the emulation based ETPC method (ETPC-emulation) proposed in our previous work~\cite{AR-PT:2023} and with the typical ZOH based event-triggered control method (ETC-ZOH). 
In the emulation based ETPC method, we consider the same parameterized control law~\eqref{eq:control_law} and the event-triggering rule~\eqref{eq:ETR}. However, at each triggering instant, the coefficients of the parameterized control law are updated by solving the following optimization problem.
\begin{equation}\label{eq:a_emu}
\begin{aligned}
\param(k)  \in & \argmin_{a \in \real^{m(p+1)}} \quad \sum_{t=t_k}^{t_k+N}
\norm{u(t)-K\hat{x}(t)}^2,\\
& \textrm{s.t.} \ \
\hat{x}(t+1) =  (A+BK)\hat{x}(t), \ \hat{x}(t_k)=x(t_k),\\ & u(t)=\mathbb{P}(t-t_k)a, \ \forall t \in [t_k,t_k + N]_\ints, 
\\ & \mathbb{P}(0)a=Kx(t_k).
\end{aligned}
\end{equation} 
In ETC-ZOH method, the control input to the plant is held constant between two successive communication time instants, i.e., $u(t)=u_k, \ \forall t \in [t_k,t_{k+1})_{\ints}$. We use the same event-triggering rule~\eqref{eq:ETR} to determine the sequence of communication time instants and at each communication time instant the control input to the plant is updated by solving the following optimization problem,
\begin{equation}\label{eq:u_ZOH}
\begin{aligned}
u_k \in & \argmin_{u \in \real^{m}} \quad \sum_{t=t_k}^{t_k+N}
\left[V(\hat{x}(t))+u^TRu\right],\\
& \textrm{s.t.} \ \
\hat{x}(t+1) =  A\hat{x}(t)+Bu, \ \forall t \in [t_k,t_k + N]_\ints, \\ & \ \hat{x}(t_k)=x(t_k),
\\ & V(\hat{x}(t)) \le \alpha^{t-t_k} V(\hat{x}(t_k)),
\ \forall t \in [t_k,t_k + M]_\ints.
\end{aligned}
\end{equation} 
\begin{figure}[h]
	\centering
	\begin{subfigure}{4cm}
		\includegraphics[width=4.4cm]{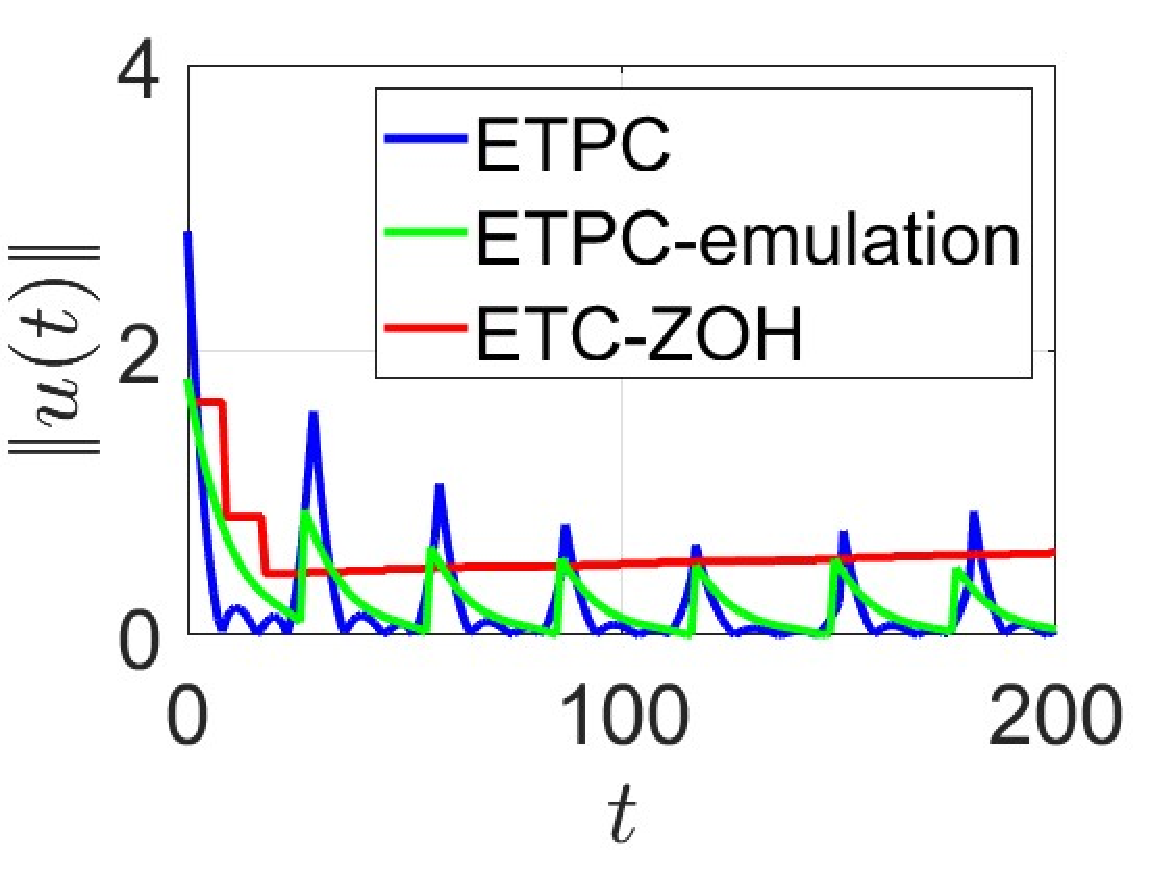}
		\caption{Control input}
		\label{fig:u}
	\end{subfigure}
	\quad
	\begin{subfigure}{4cm}
		\includegraphics[width=4.4cm]{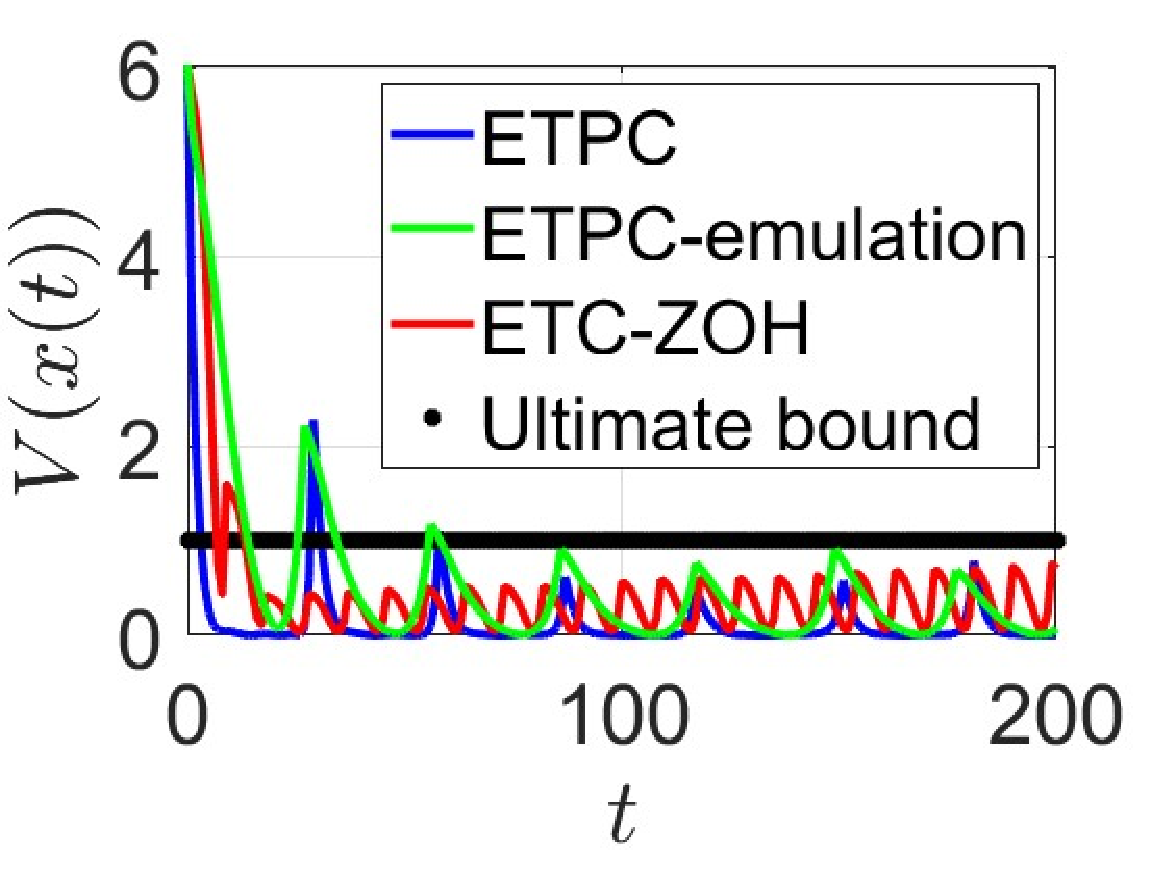}
		\caption{Convergence of $V(x)$}
		\label{fig:V}
	\end{subfigure}
	\caption{Simulation results of Example 1 for $p=3$, 
		$N=25$ and $x(0)=[2\quad5\quad6]^{\top}$.}
	\label{fig:Ex1}
\end{figure}

Figure~\ref{fig:Ex1} presents the simulation results with $p=3$, 
$N=25$ and $x(0)=[2\quad5\quad6]^{\top}$. Figure~\ref{fig:u} presents the evolution of norm of $u(t)$ and it shows that the proposed ETPC-CLF method offers smaller values for $\norm{u(t)}$ at most of the time instants compared to the other two methods. Figure~\ref{fig:V} presents the evolution of $V(x)$ along the system trajectory and it shows that $V(x)$ converges to the ultimate bound 
$\epsilon^2=1$ in all the three cases. 
Even though $u(t)$ and $V(x(t))$ are discrete-time signals, for ease of visualization, we plot them as continuous-time signals.

Next, we consider $100$ initial conditions uniformly sampled from a sphere with a specific radius and we calculate the average inter-event time (AIET) and 
the minimum inter-event time (MIET) over $100$ events for each initial 
condition with $N=30$, and $p=3$. These observations are tabulated in Table~\ref{tab:table1}.
\begin{table}[h!]
		\begin{center}
			\caption{Average of AIET and minimum of MIET, over a set of initial 
			conditions, for ETPC-CLF, ETPC-emulation and ETC-ZOH with $N=30$ and $p=3$.}
			\label{tab:table1}
			\begin{tabular}{|c|c|c|} 
				\hline
			  & 
			 \textbf{Average of AIET} & 
			 \textbf{Minimum of MIET}
			\\
				\hline
			\textbf{ETPC-CLF} & 35.2348	 & 32	\\ \hline
				\textbf{ETPC-emulation} & 25.7287	& 25	\\ \hline
					\textbf{ETC-ZOH} & 9.2121	 & 2	\\ \hline
						\end{tabular}
		\end{center}
	\end{table}
Note that,
for the given choice of control law and the event-triggering rule, the proposed
ETPC-CLF method performs better, in terms of the AIET and
MIET, compared to the ETC-ZOH method and ETPC-emulation based method. Note also that, in the ETPC-CLF method, both the AIET and the MIET are greater than $N$. This shows that the proposed method performs better, in terms of the AIET and
MIET, compared to the event-triggered model predictive control (ET-MPC) method in which the maximum inter-event time is typically chosen as the prediction horizon length $N$.

We 
repeat the procedure for the proposed ETPC-CLF method for different values of 
$N$ and $p$, and the observations are tabulated in 
Table~\ref{tab:table2}. In Table~\ref{tab:table2}, we can see that 
there is an increasing trend in the values of AIET and MIET as 
$N$ or $p$ increases. Note that as $N$ increases, the finite horizon length of the optimization problem~\eqref{eq:a_k} increases and hence leads to larger inter-event time. This is an advantage compared to the ETPC-emulation method proposed in~\cite{AR-PT:2023} where there is a decreasing trend in inter-event times as $N$ increases. Note also that choosing a larger p helps to choose a control input from a larger input space and hence
leads to better performance.

\begin{table}[h!]
	\begin{center}
		\caption{Average of AIET and minimum of MIET, over a set of initial conditions, for ETPC-CLF for 
				different values of $N$ and $p$.}
		\label{tab:table2}
		\begin{tabular}{|c|c|c|c|c|c|c|} 
			\hline
			& \multicolumn{6}{c}{N}\vline\\
			\hline
			& \multicolumn{2}{c}{10} \vline& \multicolumn{2}{c}{20} \vline& \multicolumn{2}{c}{30}\vline\\
			\hline
			p & \textbf{AIET} & \textbf{MIET} & \textbf{AIET} & \textbf{MIET} & \textbf{AIET} &\textbf{MIET} \\
			\hline
			2 & 14.1268 & 13 & 23.6609 & 23 & 33.1982 & 31\\
			3 & 15.2476 & 15 & 24.6605 & 23 & 35.2348 & 32\\
			4 &  16.0283&16  & 25.3624 & 25 & 36.62655 & 33\\
			\hline
		\end{tabular}
	\end{center}
\end{table}

\section{CONCLUSION}\label{sec:conclusion}

In this paper, we proposed an event-triggered parameterized control 
method using a control Lyapunov function approach for discrete time linear systems with external disturbances.
We designed a parameterized control law and an event-triggering rule 
that guarantee global uniform ultimate boundedness of the trajectories 
of the closed loop system and non-trivial inter-event times. 
We illustrated our 
results through numerical examples. We also showed that, for the given choice of control law and event-triggering rule, the proposed control method performs better in terms of the AIET and the MIET compared to other existing methods such as emulation based ETPC, ZOH based ETC and event-triggered MPC.

\bibliographystyle{IEEEtran}
\bibliography{references}
\end{document}